\documentclass[a4paper,11pt]{article}
\usepackage{amsfonts,amsmath,amssymb,amsthm,enumerate,bm}
\usepackage[a4paper,text={6.5in,9in},centering]{geometry}
\usepackage{xcolor}
\allowdisplaybreaks[4] 
\theoremstyle{plain}
\newtheorem{theorem}{\noindent\bf Theorem}[section]
\newtheorem{corollary}[theorem]{\noindent\bf Corollary}
\newtheorem{proposition}[theorem]{\noindent\bf Proposition}
\newtheorem{lemma}[theorem]{\noindent\bf Lemma}

\theoremstyle{remark}

\newtheorem{remark}[theorem]{\noindent\bf Remark}

\numberwithin{equation}{section}

\def\be{\begin{eqnarray}}
\def\ee{\end{eqnarray}}
\def\ben{\begin{eqnarray*}}
\def\een{\end{eqnarray*}}
\def\benum{\begin{enumerate}}
\def\eenum{\end{enumerate}}
\newcommand{\lr}{\left(}
\newcommand{\rr}{\right)}

\newcommand{\lge}{\left\{ }
\newcommand{\rge}{\right\} }

\title{\bf {Rubio de Francia Extrapolation Theorem for Quasi non-increasing Sequences} }
\author{
{\bf Monika Singh} \\
Department of Mathematics\\
Lady Shri Ram College for Women
(University of Delhi)\\
Lajpat Nagar, New Delhi-110 024, India\\
 {\bf Email:} monikasingh@lsr.du.ac.in\\[3mm]
 {\bf Amiran Gogatishvili\footnote{Corresponding author}}\\
Institute of Mathematics of the Czech Academy of Sciences\\
Zitna 25, 110 00 Prague 1\\ Czech Republic \\
 {\bf Email:} gogatish@math.cas.cz\\[3mm] 
 {\bf Rahul Panchal} \\
Department of Applied Science and Humanities\\
School of Engineering and Technology\\ 
Vivekananda Institute of Professional Studies - Technical Campus\\ 
Pitampura, New Delhi, 110 034, India\\
{\bf Email:}   rahulpanchalmaths@gmail.com\\[3mm]
{\bf Arun Pal Singh}\\
Department of Mathematics \\
Dyal Singh College
(University of Delhi)\\ Lodhi Road, New Delhi- 110 003, India\\
{\bf Email:} arunpalsingh@dsc.du.ac.in \\[3mm]
}

\date{}

\begin{document}
\maketitle

\begin{abstract}
\noindent We prove the discrete Rubio de Francia extrapolation theorem for a pair of quasi non-increasing sequences with $\mathcal{QB}_{\beta, p}$ weight class. Also, a weight  characterization of the boundedness of the generalized discrete Hardy averaging operator on the class of quasi non-increasing sequences from $l_w^p(\mathbb{Z}^+)$ is proved.


\end{abstract}


     \section{Introduction}

     By a weight sequence $\{w(k)\}_{k \in \mathbb Z^+},$ we mean a non-negative sequence of real numbers. We say that a weight sequence $\{w(k)\}_{k \in \mathbb Z^+}$ is in the Ari\~{n}o -Muckenhoupt discrete weight class $\mathcal B_p$ (see \cite{bterd}),  if there exists a constant $c>0$ such that
\[
\sum_{k=n}^\infty \left( \frac{n}{k}\right)^pw(k) \le c \sum_{k=1}^n w(k) ~ \text{for all $n \in \mathbb{Z}^+$}.
\]
The discrete $\mathcal B_p$ class of weights is one of the important weight class as it characterizes the boundedness of the  discrete Hardy averaging operator 
\begin{equation}\label{s3}
		(\mathcal{A}y)(n) := \frac{1}{n}\sum_{k=1}^{n} y(k),
	\end{equation}
on a class of non-negative non-increasing sequences  $\{y(k)\}_{k \in \mathbb Z^+}$.  Precisely, we have the following:

\vspace{2pt}

\noindent {\bf Theorem A. \cite{bterd}} \emph{Let $1\le p<\infty.$   Then there is a constant $C>0$ such that}
	    	\be \label{12}
	    	\nonumber \sum_{k=1}^{\infty}(\mathcal{A}y)^p(k) w(k) \le C\sum_{k=1}^{\infty}y^p(k) w(k)
	    	\ee
\emph{holds for all non-negative non-increasing sequences $\{y(k)\}_{k \in \mathbb Z^+}$ if and only if $\{w(k)\}_{k \in \mathbb Z^+}\in \mathcal{B}_p.$}
 
\vspace{2pt}

Recently, in \cite{sa}, Saker and Agarwal have proved the \emph{discrete Rubio de Francia extrapolation theorems} for the class of non-negative non-increasing sequences by using the discrete weight class $\mathcal B_p$.
Precisely, they prove the following:

\vspace{2pt}

\noindent {\bf Theorem B. \cite{sa}} \emph{Let $\varphi$ be an increasing function on $(0, \infty)$ and  $0 < p_0 < \infty.$ Suppose that for every $\{w(k)\}_{k \in \mathbb Z^+} \in \mathcal B_{p_0} $} 
\be 
	\nonumber \sum_{k=1}^{\infty} f^{p_0}(k)w(k)\le \varphi (\mathcal{B}_{p_0}(w))\sum_{k=1}^{\infty} g^{p_0}(k)w(k)
	\ee
\emph {holds for every non-negative non-increasing sequences $\{f(k)\}_{k \in \mathbb Z^+}~ and ~\{g(k)\}_{k \in \mathbb Z^+}$. Then, for every $0 < p < \infty$ and $\{w(k)\}_{k \in \mathbb Z^+} \in \mathcal{B}_p$ the following holds}
\be 
	\nonumber\sum_{k=1}^{\infty} f^p(k)w(k)\le \varphi^* (\mathcal{B}_p(w)) \sum_{k=1}^{\infty} g^p(k)w(k),
	\ee
	{\it where $\varphi^* (\mathcal{B}_p(w))$ is as given in} \cite{sa} {\it and $\mathcal{B}_p(w)$ is a constant defined as }
\[
\mathcal{B}_p(w) := \inf \left\{ c \ge 1 :\sum_{k=1}^n w(k) + \sum_{k=n}^\infty \left(\frac{n}{k}\right)^{ p}w(k)\le c\sum_{k=1}^n w(k), ~  \text{for all}~~  n \in \mathbb Z^+ \right\}.\]
 
The continuous version of Theorem B was proved by Carro and Lorente  in 2010 \cite{cl} 
for non-negative non-increasing measurable functions with $B_p$ class of weights. The  $B_p$ weight class is a continuous version of the discrete $\mathcal B_p$ class of  weight sequences. A weight $w$ is said to belong to the class $B_p ~(p>0),$  if there exists a constant $c>0$ such that the inequality 
\[
\int_r^\infty \lr\frac{r}{x}\rr^p w(x)dx \leq c\int_0^r w(x)dx
\]
holds for every $r>0.$ Let us write
\begin{equation*}\label{e1.1}
[w]_{B_p} := \inf \left\{ c \ge 1 :  \int_0^r w(x)dx+\int_r^\infty \lr\frac{r}{x}\rr^p w(x)dx \le c\int_0^r w(x)dx, \,r>0 \right\}.
\end{equation*}
One of the important properties of the $B_p$class of weights (see \cite{cl, crs}) is an open ended property which states that: if $w\in B_p ~(p>0),$ then there exists $\varepsilon>0$ such that $w\in B_{p-\varepsilon}.$ Moreover,
\begin{equation}\label{e11}
[w]_{B_{p-\varepsilon}}\le \frac{c[w]_{B_p}}{1-\varepsilon \alpha^p [w]_{B_p}}
\end{equation}
where $0<\alpha <1$ is the universal constant and $\varepsilon$ is such that $1-\varepsilon \alpha^p [w]_{B_p}>0,$ which is instrumental in proving the Carro and Lorente extrapolation result given below.

\vspace{2pt}

\noindent {\bf Theorem C. \cite{cl}}  {\it
Let $\varphi$ be an increasing  function defined on $(0,\infty),~ (f,g)$ be a pair of positive non-increasing  functions defined on $(0,\infty)$ and $0 < p_0<\infty.$ Suppose that for every $w\in B_{p_0},$ the inequality
\[
\int_0^\infty f^{p_0}(x)w(x)dx \leq \varphi([w]_{B_{p_0}})\int_0^\infty g^{p_0}(x)w(x)dx
\]
holds. Then for all $0 < p<\infty$ and $w \in B_{p},$ the following holds:
\[
\int_0^\infty f^p(x)w(x)dx \leq \tilde\varphi([w]_{B_p})\int_0^\infty g^p(x)w(x)dx,
\]
where 
$$
\tilde\varphi([w]_{B_p}) := \inf_{0<\varepsilon<\frac{p_0}{p\alpha^p [w]_{B_p}}} \varphi \lr \frac{p_0}{\varepsilon} \rr^{p/p_0} \frac{c[w]_{B_p}}{1-  \frac{\varepsilon p}{p_0}  \alpha^p [w]_{B_p}      }
$$
with $c$ as in \eqref{e11}.
}\\
An important intermediary to establish Theorem C has been the boundedness of the generalized Hardy averaging operator for non-negative non-increasing functions, see \cite{cl, cs}.\\

The results obtained in \cite{cl} are parallel to those of the extrapolation results of Rubio de Francia for non-negative measurable functions done with the weights from the $A_p$ class of weights (see \cite{mu2}), popularly known as the Muckenhoupt weight class. It is an important weight class, which characterizes the boundedness of the Hardy Littlewood maximal operator  \cite{mu2}, and the Riesz potential operator  \cite{m} in Lebesgue spaces and in grand Lebesgue spaces \cite{fgj, is, m}. In fact, it is due to the Spanish mathematician Jose L. Rubio de Francia who initiated and developed an  extrapolation theory in Lebesgue spaces during 1982-85 \cite{rdf, rdf1, rdf2}, which is now known after his name. For a detailed study of this theory, see \cite{dc, dcump1, dcump2}. \\

In 1994, Bergh, Burenkov and Persson \cite{bbp} investigated  Hardy's inequality for quasi-monotone functions on Lebesgue spaces with power-type weights.  For $~\beta\in \mathbb{R},$ a measurable function $f$ is said to belong to $Q_\beta,$ \emph{a class of quasi non-negative  non-increasing  functions},  if $x^{-\beta}f(x)$ is non-increasing. Clearly, when $\beta=0, $ it becomes the class of non-negative non-increasing functions. 
In \cite{malig}, the authors have exclusively studied weighted inequalities for quasi monotone functions. Later, the Hardy's inequality for the class of quasi non-negative non-increasing functions with the general weights was proved in \cite{pma}. The following was proved:

\vspace{2pt}

 \noindent {\bf Theorem D. \cite{pma}} \emph{ For $~1\leq p<\infty ,~\beta >-1,$ the inequality
\[
\lr \int_0^\infty {\lr \frac{1}{x}\int_0^x f(t)dt \rr}^p w(x)dx\rr \leq C\int_0^\infty f^p(x) w(x)dx, 
\]
holds for all $f\in Q_{\beta}$ if and only if $w\in QB_{\beta,p},$ i.e., the weight $w$ satisfies the condition:
\be \label{eqn aa}
\nonumber\int_r^\infty \lr\frac{r}{x}\rr^p w(x)dx \leq C\int_0^r {\lr \frac{x}{r} \rr}^{\beta p} w(x)dx,~ r>0.
\ee}

Note that for $\beta = 0,$ the weight class $QB_{\beta,p}$ reduces to the class $B_p.$ 
In 2023, the authors \cite{arpm} extended the Carro and Lorente extrapolation result to the class of non-negative quasi non-increasing functions.

Motivated by the above works, our main objective in this paper is to prove the Rubio de Francia extrapolation theorem for non-negative, quasi non-increasing sequences. To begin with, in Section 2, we provide some preliminaries and known results that will be used in the sequel and establish Hardy's inequality for the generalized Hardy averaging operator acting on non-negative, quasi non-increasing sequences. In Section 3, we establish an open-ended property for the weight class $\mathcal{QB}_{\beta, p},$ which is used to finally establish the Rubio de Francia extrapolation result for quasi non-increasing sequences.

\section{Discrete generalized Hardy averaging operator}
In this section, we give  some basic terminologies and lemmas that will be used throughout the paper. There after, we give the characterization for the boundedness of the discrete generalized Hardy operator in the framework of the Lebesgue sequence spaces for the class of quasi non-increasing sequences. We begin with the following:\\
For $\beta \ge -1$ {and $p>0,$ \emph{a weight sequence $\{w(k)\}_{k \in \mathbb Z^+}$ is said to be in the class $ \mathcal{QB}_{\beta, p}$} if for some $c>0,$ it satisfies the following inequality 
\[
\sum_{k=n}^\infty \left( \frac{n}{k}\right)^pw(k) \le c \sum_{k=1}^n \left(\frac{k}{n}\right)^{\beta p}w(k) ~ \text{for all $n \in \mathbb{Z}^+$}.
\] 

\noindent If  $\beta=0,$ we shall denote $\mathcal{B}_{0,p}=:\mathcal{B}_{p},$  the Ari\~{n}o -Muckenhoupt class of weights. Note that  
$$\{w(k)\}_{k \in \mathbb Z^+} \in \mathcal{QB}_{\beta, p} \Leftrightarrow \{k^{\beta p}w(k)\} \in \mathcal{B}_{(\beta +1) p} .$$ 

\noindent For a weight sequence $\{w(k)\}_{k \in \mathbb Z^+} \in \mathcal{QB}_{\beta, p},$ we define it's $\mathcal{QB}_{\beta, p}$ constant as follows :
\[
[w]_{\mathcal{QB}_{\beta, p}} := \inf \left\{ d \ge 1 :\sum_{k=1}^n\left(\frac{k}{n}\right)^{\beta p}w(k) + \sum_{k=n}^\infty \left(\frac{n}{k}\right)^{ p}w(k)\le d\sum_{k=1}^n\left(\frac{k}{n}\right)^{\beta p}w(k)\right\}.\]
\noindent For $\beta \ge -1,$ a sequence $\{f(k)\}_{k\in \mathbb{Z}^+}$ is said to be in the class $ \mathcal{Q}_\beta$ if the sequence $\{k^{-\beta}f(k)\}_{k\in \mathbb{Z}^+}$ is non-increasing, we shall call such sequences to be \emph{quasi non-increasing} sequences. Quasi non-increasing sequences were initially defined in \cite{robert, shah}.

\vspace{2 pt}

\noindent For a given non-negative sequence  $\{\psi (k)\}_{k\in \mathbb{Z}^+},$  the \emph{discrete generalized Hardy averaging operator} is defined as
\begin{equation*}
	(\mathcal{A}_\psi f)(k) := \frac{1}{\Psi (k)}\displaystyle\sum_{\tau=1}^{k}f(\tau)\psi (\tau),
\end{equation*}
where $\Psi (k):= \displaystyle \sum_{\tau=1}^k \psi (\tau)$ and $\{f(k)\}_{k\in \mathbb{Z}^+} $ is a non-negative sequence.

\noindent Throughout this paper, we will often rely on the lemmas stated below. 

\vspace{2 pt}

\noindent{\bf Lemma E. \cite{gerd} (Power Rule I)}
	\emph{Let  $0 \le q < \infty$ and $\{f(k)\}_{k\in \mathbb{Z}^+}$ be a non-negative sequence. Then the following holds:}
$$	\min\{1,q\} \sum_{k=1}^{n}f(k)\left( \sum_{j=1}^{k}f(j)\right)^{q-1} \le \left(\sum_{k=1}^{n}f(k)\right)^q \le \max\{1,q\}\sum_{k=1}^{n}f(k)\left( \sum_{j=1}^{k}f(j)\right)^{q-1}
	$$
	for $n\in \mathbb{Z}^+.$
	
\vspace{2pt}

\noindent In particular, taking $\{f(k)\}_{k\in \mathbb{Z}^+} = \{1\}_{k\in \mathbb{Z}^+}$ in Lemma E (Power Rule), we get the following estimate:
$$
\min\{1,q\}\sum_{k=1}^{n}k^{q-1} \le n^q \le \max\{1,q\}\sum_{k=1}^{n}k^{q-1}.
$$

\vspace{0.1cm} 

\noindent{\bf Lemma F. \cite{GBKKU} (Power Rule II)}
	\emph{Let  $0 \le q <\infty$ and $\{f(k)\}_{k\in \mathbb{Z}^+}$ be a non-negative sequence such that $\sum_{j=1}^{\infty }f(j) =\infty$. Then the following holds:}
$$	\min\lge1,\frac{1}{q}\rge \left(\sum_{k=1}^{n}f(k)\right)^{-q}\le   \sum_{k=n}^{\infty}f(k+1)\left( \sum_{j=1}^{k}f(j)\right)^{-q} \left( \sum_{j=1}^{k+1}f(j)\right)^{-1}\le  \max\lge 1,\frac{1}{q}\rge\left(\sum_{k=1}^{n}f(k)\right)^{-q}
	$$
	for $n\in \mathbb{Z}^+.$\\
In particular, taking  $\{f(k)\}_{k\in \mathbb{Z}^+} = \{1\}_{k\in \mathbb{Z}^+}$ in Lemma F, we obtain the following:
$$
\min\lge 1,\frac{1}{q}\rge n^{-q} \le \sum_{k=n}^{\infty} \frac{1}{k^{q}(k+1)} \le\max\lge1,\frac{1}{q}\rge n^{-q}.
$$

We now prove a lemma which is of independent interest.

\begin{lemma} 
For $n<k$, $m\ge0,$ the following holds:
$$
\ \frac{1}{2(m+1)}\left(\ln\left(\frac{k}{n}\right)\right)^{m+1}  \le \sum_{i=n+1}^{k} \frac{ \left(\ln \left(\frac{i}{n}\right)\right)^{m}}{i} \le \frac{2^{m+2}}{(m+1)}\left(\ln\left(\frac{k}{n}\right)\right)^{m+1}.
$$
\end{lemma}
\begin{proof} 
Left estimate:
  \begin{align*} \sum_{i=n+1}^{k} \frac{ (\ln \left(\frac{i}{n}\right))^{m}}{i} &\ge  \frac{1}{2}\sum_{i=n+1}^{k} \frac{ (\ln\left(\frac{i}{n}\right))^{m}}{i-1} \int_{i-1}^{i} dt\\
  & \ge \frac{1}{2} \sum_{i=n+1}^{k} \int_{i-1}^i \frac{ (\ln \left(\frac{t}{n}\right))^{m}}{t} dt\\
  &= \frac{1}{2} \int_{n}^k \frac{ (\ln \left(\frac{t}{n}\right))^{m}}{t} dt 
   = \frac{1}{2(m+1)} \left(\ln\left(\frac{k}{n}\right)\right)^{m+1}.
   \end{align*}
   \noindent Right estimate:\\
    As $n+1\le k $, we have that $n(k+1)\le k(n+1)\le k^2$. 
 Therefore,
\[ \left(\ln \left(\frac{k+1}{n}\right)\right)^{m+1}\le
2^{m+1}\left(\ln \left(\frac{k}{n}\right)\right)^{m+1}.
\]
 
    \begin{align*} \sum_{i=n+1}^{k} \frac{ (\ln  \left(\frac{i}{n}\right) )^{m}}{i} &=  \sum_{i=n+1}^{k} \frac{ \left(\ln \left(\frac{i}{n}\right)\right)^{m}}{i} \int_{i}^{i+1} dt\\
  & \le  2 \sum_{i=n+1}^{k} \int_{i}^{i+1} \frac{ (\ln  \left(\frac{t}{n}\right))^{m}}{t} dt\\
  &\le 2 \int_{n}^{k+1} \frac{ (\ln  \left(\frac{t}{n}\right))^{m}}{t} dt 
   = \frac{2}{(m+1)} \left(\ln\left(\frac{k+1}{n}\right)\right)^{m+1}\\
   & \le  \frac{2^{m+2}}{(m+1)} \left(\ln\left(\frac{k}{n}\right)\right)^{m+1}.
   \end{align*}

\end{proof}

\vspace{0.1cm}

\noindent{\bf Lemma G. \cite{gerd} (Partial Sums Lemma)}\label{lem2}
		\emph{Let  $\{f(k)\}_{k\in \mathbb{Z}^+}$ and $\{g(k)\}_{k\in \mathbb{Z}^+}$ be two sequences of non-negative real numbers such that
		\[
		\sum_{k=1}^{n}f(k) \le \sum_{k=1}^{n}g(k)
		\]
		for all $n\in \mathbb{Z}^+,$ then the following estimate holds:
		\[
		\sum_{k=1}^{n}f(k)\varphi(k) \le \sum_{k=1}^{n}g(k)\varphi(k)
		\]
		for all $n\in \mathbb{Z}^+$ and each non-negative, non-increasing sequence $\{\varphi(k)\}_{k\in \mathbb{Z}^+}.$}
        
\vspace{2 pt}

\noindent{\bf Lemma H. \cite{smk}} ({\bf Fubini Theorem})\label{lem4}
\emph{Assume that  $\{f(k)\}_{k\in \mathbb{Z}^+}$ and  $\{g(k)\}_{k\in \mathbb{Z}^+}$ are two sequences of non-negative real numbers. Then }
\begin{equation*}\label{eq04}
	\sum_{k=1}^{n}f(k) \sum_{j=k}^{n}g(j)=\sum_{j=1}^{n}g(j) \sum_{k=1}^{j}f(k)
\end{equation*}
\emph{for all} $n\in \mathbb{Z}^+.$
\begin{lemma}\label{lm12}
Let $\beta\ge0$ be given. If a non-negative sequence $\{\psi (n)\}_{n\in \mathbb{Z}^+}$  satisfies the condition
	\begin{equation}\label{eqa91}
		\sum_{k=1}^n \psi(k) \le c\sum_{k=n}^{2n} \psi(k),
	\end{equation}
 then the following  holds 
 \begin{equation}\label{eqa92}
		n^\beta\sum_{k=1}^n \psi(k) \le C \sum_{k=1}^n k^\beta \psi(k).
	\end{equation}  
\end{lemma}

\begin{proof}
Suppose \eqref{eqa91} holds. For  $n=1,2,3$, it is trivial. For  $n\ge 4,$ using the estimate, 
$$\frac{n}{4} \le \left[\frac{n}{2}\right]\le \frac{n}{2}\le\left[\frac{n}{2}\right]+1,$$
we get
\begin{align*}
		n^\beta\sum_{k=1}^n \psi(k) &= n^\beta\sum_{k=1}^{\left[\frac{n}{2}\right]} \psi(k) +n^\beta\sum_{k=\left[\frac{n}{2}\right]+1}^n \psi(k)\\
        &\le  c n^\beta\sum_{k={\left[\frac{n}{2}\right]}}^{2\left[\frac{n}{2}\right]} \psi(k) +2^\beta \sum_{k=\left[\frac{n}{2}\right]+1}^n k^\beta\psi(k)\\
         &\le  c n^\beta\sum_{k={\left[\frac{n}{2}\right]}}^{n} \psi(k) +2^\beta \sum_{k=\left[\frac{n}{2}\right]+1}^n k^\beta\psi(k)\\  
          &\le  c 4^\beta\sum_{k={\left[\frac{n}{2}\right]}}^{n} k^\beta\psi(k) +2^\beta \sum_{k=1}^n k^\beta\psi(k)\\ 
          &\le (4^\beta c+2^\beta) \sum_{k=1}^n k^\beta \psi(k)=  C \sum_{k=1}^n k^\beta \psi(k),
\end{align*}
where $C=(4^\beta c+2^\beta)$.\\
\end{proof}

\begin{lemma}
Let $\beta\ge0$ be given. If a non-negative sequence $\{\psi (n)\}_{n\in \mathbb{Z}^+}$  satisfies the condition \eqref{eqa92} then there exists $m \in \mathbb Z^+$  such that 
  \begin{equation}\label{eqa94}
		\sum_{k=1}^n \psi(k) \le c\sum_{k=n}^{mn} \psi(k).
	\end{equation}
 
\end{lemma}
\begin{proof}
Suppose $\{\psi (n)\}_{n\in \mathbb{Z}^+}$  satisfies the condition \eqref{eqa92}. Let $m \in \mathbb Z^+$ such that $C<m^{\beta}.$ Considering \eqref{eqa92} for an integer $mn,$ we have
\begin{align*}
		 (mn)^{\beta}\sum_{k=1}^{mn} \psi(k) &\le C \sum_{k=1}^{mn}k^\beta \psi(k)\\
        &=  C\lr\sum_{k=1}^{n} k^\beta\psi(k) +  \sum_{k=n+1}^{mn} k^\beta\psi(k)\rr\\
         &\le  C \lr n^{\beta}\sum_{k=1}^{n} \psi(k) +(mn)^{\beta} \sum_{k=n+1}^{mn} \psi(k) \rr.
        \end{align*}
        Therefore, we have

 \begin{equation}\label{eqa95}
   m^{\beta}\sum_{k=1}^{mn} \psi(k)  \le   C \lr\sum_{k=1}^{n} \psi(k) +m^{\beta} \sum_{k=n+1}^{mn} \psi(k) \rr.
 \end{equation}
The inequality \eqref{eqa95} can be written as
     \begin{align*}
		\sum_{k=1}^{n} \psi(k) &\le m^{\beta}\lr \frac{C-1}{m^{\beta}-C}\rr\sum_{k=n}^{mn} \psi(k).
        \end{align*}  
        Thus, we obtain \eqref{eqa94} for $c=  m^{\beta}\lr \frac{C-1}{m^{\beta}-C}\rr.$
   \end{proof}
\begin{remark}
Towards the converse of  Lemma \ref{lm12}, if \eqref{eqa92} holds, then we cannot assert that   \eqref{eqa91} holds. However, we can guarantee the existence of a positive integer $m$ such that \eqref{eqa94} holds. 
   \end{remark}
In  \cite{rma}, authors have proved different necessary and sufficient conditions for the boundedness of the generalized Hardy averaging operator  on subclass of the quasi non-increasing sequences, but with  stronger assumptions. Here we give a characterization for the boundedness of the generalized Hardy averaging operator acting on quasi non-increasing sequences. Following is the main result of this section.
\begin{theorem}\label{th1}
	Let  $0 < p<\infty, ~\beta\ge0$ and $\{\psi (n)\}_{n\in \mathbb{Z}^+}$ be a given non-negative sequence such that \eqref{eqa91} holds.
	Then the inequality 
	\begin{equation}\label{eqa2}
		\sum_{n=1}^\infty(\mathcal{A}_\psi y)^p(n)v(n) \le C\sum_{n=1}^\infty y^p(n)v(n)
	\end{equation}
	holds for all non-negative sequences $\{y(n)\}_{n\in \mathbb{Z}^+}\in \mathcal{Q}_\beta$ if and only if the weight sequence $\{v(n)\}_{n\in \mathbb{Z}^+}$ satisfies
	\begin{equation}\label{eqa3}
		\left(\sum_{k=1}^n k^\beta \psi(k)\right)^p \sum_{k=n}^\infty \Psi^{-p}(k)v(k) \le C \sum_{k=1}^n k^{\beta p}v(k).
	\end{equation}
\end{theorem}
\begin{proof}
Suppose  \eqref{eqa2} holds. For given $ n \in \mathbb Z^+,$ take
\begin{equation*}
    y(k):=
    \begin{cases}
        k^\beta, & k=1,2,3,\cdots, n;\\
        0, & k>n.
    \end{cases}
\end{equation*}
Clearly, $ \{y(k)\}_{ k \in \mathbb Z^+} \in \mathcal Q_\beta.$ Then, on computing  \eqref{eqa2} for $\{y(k)\}_{ k \in \mathbb Z^+},$ we easily get \eqref{eqa3}. 

\vspace{1.2 pt}

\noindent Towards the converse, suppose that the conditions \eqref{eqa91} and \eqref{eqa3} hold. Clearly, by Lemma \ref{lm12}, the estimate \eqref{eqa92} holds.\\
We first consider the case $p>1$. 
   Let $\{y(n)\}_{n \in \mathbb Z^+}\in \mathcal{Q}_\beta$ be any  non-negative sequence.  Then by using the Power Rule I, Fubini Theorem, Partial Sums Lemma, conditions \eqref{eqa3}, \eqref{eqa92} and the  H\"older inequality, we have 
  \begin{align*}
		\sum_{n=1}^\infty(\mathcal{A}_\psi y)^p(n)v(n) &\le p  \sum_{n=1}^\infty\left(\sum_{k=1}^n \left (\sum_{j=1}^ky(j)\psi(j)\right)^{p-1} y(k)\psi(k)\right)\Psi(n)^{-p} v(n)\\
        &=p\sum_{k=1}^\infty \left (\sum_{j=1}^ky(j)\psi(j)\right)^{p-1} y(k)\psi(k)\sum_{n=k}^\infty \Psi(n)^{-p} v(n)\\
        &= p\sum_{k=1}^\infty \left (\frac{1}{\sum_{j=1}^{k} j^{\beta }\psi(j)}\sum_{j=1}^ky(j)j^{-\beta}j^{\beta}\psi(j)\right)^{p-1} k^{-\beta}y(k) \\
        &\hskip+2cm \times \left(\sum_{j=1}^{k}j^{\beta }\psi(j)\right)^{p-1}k^{\beta} \psi(k)\sum_{n=k}^\infty \Psi(n)^{- p} v(n)\\
        &\le C p\sum_{k=1}^\infty \left (\frac{1}{\sum_{j=1}^{k} j^{\beta }\psi(j)}\sum_{j=1}^ky(j)j^{\beta}j^{-\beta}\psi(j)\right)^{p-1} k^{-\beta}y(k) k^{\beta p}v(k)\\
        &=C p\sum_{k=1}^\infty (A_\Psi y)^{p-1}(k)\left (\frac{k^\beta\sum_{j=1}^{k}\psi(j)}{\sum_{j=1}^{k} j^{\beta }\psi(j)}\right)^{p-1} y(k) v(k)\\
         &\le C (4^\beta c+2^\beta)^{p-1} p\sum_{k=1}^\infty (A_\Psi y)^{p-1}(k) y(k) v(k)\\ 
         &\le C_1\left(\sum_{k=1}^\infty (A_\Psi y)^{p}(k)v(k)\right)^{\frac{1}{p'}} \left(\sum_{k=1}^\infty y(k)^p v(k) \right)^{\frac{1}{p}},
        \end{align*}
        where $C_1=C p (4^\beta c+2^\beta)^{p-1}.$\\
Now we consider the case  $0<p\le 1$.  Let  $\{y(n)\}_{n \in \mathbb Z^+}$ be any non-negative sequences in $ \mathcal{Q}_\beta$, By using the condition  \eqref{eqa92},
  the nth term of the sequence $\{y(n)\}_{n \in \mathbb Z^+}$ can be written as
  \begin{align*} y(n)&=n^{-\beta}y(n)\frac{n^\beta\Psi(n)\sum_{k=1}^n k^{\beta}\psi(k)}{\Psi(n)\sum_{k=1}^n k^{\beta}\psi(k)}\\
  &\le (4^\beta c+2^\beta)\frac{1}{\Psi(n)} \sum_{k=1}^n k^{-\beta}y(k)k^{\beta}\psi(k)\\
 &=  \frac{(4^\beta c+2^\beta)}{\Psi(n)} \sum_{k=1}^n y(k)\psi(k).
  \end{align*}
  Therefore,    
\begin{equation}\label{eqa96}
    \left(\sum_{k=1}^n y(k)\psi(k)\right)^{p-1} \le \lr\frac{1}{4^\beta c+2^\beta}\rr^{p-1} y(n)^{p-1}\Psi(n)^{p-1}.
\end{equation}
Now,  by considering the LHS of \eqref{eqa2}, on using  Power Rule I, \eqref{eqa96} and   Fubini Theorem  we get  
\begin{align}\label{eqa97}
		\nonumber\sum_{n=1}^\infty(\mathcal{A}_\psi y)^p(n)v(n) &\le   \sum_{n=1}^\infty\left(\sum_{k=1}^n \left (\sum_{j=1}^ky(j)\psi(j)\right)^{p-1} y(k)\psi(k)\right)\Psi(n)^{-p} v(n)\\\nonumber
        &\le \lr\frac{1}{4^\beta c+2^\beta}\rr^{p-1}  \sum_{n=1}^\infty\left(\sum_{k=1}^n y(k)^p\Psi(k)^{p-1} \psi(k)\right)\Psi(n)^{-p} v(n)\\\nonumber
        &=\lr\frac{1}{4^\beta c+2^\beta}\rr^{p-1}   \sum_{k=1}^\infty y(k)^p\Psi(k)^{p-1} \psi(k) \sum_{n=k}^\infty\Psi(n)^{-p} v(n)\\\nonumber
               &\le  \lr\frac{1}{4^\beta c+2^\beta}\rr^{p-1}   \sum_{k=1}^\infty k^{-\beta p} y(k)^p\left( \sum_{j=1}^k j^\beta\psi(j)\right)^{p-1}  \\
        &\hskip+2cm \times k^\beta\psi(k) \sum_{n=k}^\infty\Psi(n)^{-p} v(n).  
\end{align}
Since $\{k^{-\beta}y(k)\}$ is a non-increasing sequence, by using the Partial Sums Lemma  and condition \eqref{eqa3} in \eqref{eqa97}, we obtain 
     \begin{align*}
		\sum_{n=1}^\infty(\mathcal{A}_\psi y)^p(n)v(n) 
               &\le C\lr \frac{1}{4^\beta c+2^\beta}\rr^{p-1} p  \sum_{k=1}^\infty k^{-\beta p} y(k)^p k^{\beta p} v(k)\\
               &= C_2   \sum_{k=1}^\infty  y(k)^p  v(k),  
\end{align*}
where $C_2= C \lr 4^\beta c+2^\beta\rr^{1-p}.$ Hence the assertion is proved.
\end{proof}
\begin{remark}\label{r3}
Observe that the assumption \eqref{eqa91} in Lemma \ref{lm12} can be replaced by \eqref{eqa94}. By following similar steps, it can be proved that it affects only the constants appearing in the proof.  A similar change can be made in the statement of Theorem \ref{th1}.  
\end{remark}
\begin{remark}
  Note that for the case $ \beta=0,$ the condition \eqref{eqa3} reduces to the $\mathcal B_{\psi,p}$ condition, see \eqref{s5} below,  and the class of quasi non-increasing sequences reduces to the class of non-increasing sequences. In \cite{sa} and \cite{smk}, authors have characterized the boundedness of the discrete generalised Hardy averaging operator for the class of non-increasing sequences. Theorem \ref{th1} generalizes the results proved in \cite{sa, smk}.   
\end{remark}

\noindent Precisely, following is a particular case of Theorem \ref{th1}:
\begin{corollary}\label{s9}
	Let  $0 < p<\infty$ and $\{\psi (n)\}_{n\in \mathbb{Z}^+}$ be a given non-negative sequence. 
	Then the inequality 
	\begin{equation*}
		\sum_{n=1}^\infty(\mathcal{A}_\psi y)^p(n)v(n) \le C\sum_{n=1}^\infty y^p(n)v(n)
	\end{equation*}
	holds for all non-negative non-increasing sequences $\{y(n)\}_{n\in \mathbb{Z}^+}$ if and only if the weight sequence $\{v(n)\}_{n\in \mathbb{Z}^+}$ satisfies
	\begin{equation*}\label{s5}
		\sum_{k=n}^\infty\lr\frac{\Psi(n)} {\Psi(k)}\rr^{p}v(k) \le C \sum_{k=1}^n v(k),   ~~\text{ for all} ~~n \in \mathbb Z^+.
	\end{equation*}
\end{corollary}

\begin{remark}
In Corollary \ref{s9}, for $\beta \ge0$, by taking $\{\psi(k)=k^\beta\}_{ k \in \mathbb Z^+}$ $\{v(n)\}_{ n \in \mathbb Z^+}=\{v(n) n^{\beta p}\}_{ n \in \mathbb Z^+}$ and $\{y(n)\}_{ n \in \mathbb Z^+}=\{y_n n^{-\beta}\}_{ k \in \mathbb Z^+},$ on using  approximation $\Psi(n) = \displaystyle \sum_{k=1}^{n} k^\beta \approx n^{\beta +1},$ we get a characterization of the boundedness of the discrete Hardy operator (given in  \eqref{s3}), on the class of quasi non-increasing sequences.
\end{remark}

\begin{corollary}\label{s8}
Let  $0 < p<\infty.$ 
	Then the inequality 
	\begin{equation*}
		\sum_{n=1}^\infty(\mathcal{A} y)^p(n)v(n) \le C\sum_{n=1}^\infty y^p(n)v(n)
	\end{equation*}
	holds for all  $\{y(n)\}_{n\in \mathbb{Z}^+} \in \mathcal Q_\beta$ if and only if the weight sequence $\{v(n)\}_{n\in \mathbb{Z}^+}$ satisfies
	\begin{equation*}
		\sum_{k=n}^\infty\lr\frac{n} {k}\rr^{p}v(k) \le C \sum_{k=1}^n\lr\frac{k} {n}\rr^{\beta p} v(k)   ~~\text{ for all} ~~n \in \mathbb Z^+.
	\end{equation*}

\end{corollary}

\noindent The Corollary \ref{s8} can also be obtained from Theorem \ref{th1} on taking  in particular $\{\psi (n)=1\}$ for all $n \in \mathbb Z^+.$

\section{Extrapolation Theorem on $\mathcal Q_\beta$}
In this section, we shall prove the Rubio de Francia extrapolation theorem for the class of quasi non-increasing sequences. Before giving the main theorem, we state some lemmas that are needed in the proof.

\noindent{\bf Lemma I. \cite{hlp}}  \emph{If $x$ and $y$ are positive and unequal, then for $r<0$ or $r>1$,} the following holds:
\[
ry^{r-1}(x-y)<x^r-y^r<rx^{r-1}(x-y).
\]
 \noindent \emph{The above inequality reverses when} $0<r< 1.$\\
 
For $r\in \mathbb{R}$ and $\tau\in \mathbb{Z}^+,$ the forward difference operator is given by
\[
\Delta \tau^r = (\tau+1)^r - \tau^r.
\]
From the definition of the forward difference operator, we get
 
	\[
\sum_{\tau=1}^k \Delta\tau^r = (k+1)^r-1
\]
for all $r\in \mathbb{R}$ and $\tau \in \mathbb{Z}^+,$
and
\[
\sum_{\tau= k+1}^\infty \Delta\tau^{r} = -(k+1)^{r}
\]
for all $r\in \mathbb{R}, \quad r<0$ and $\tau \in \mathbb{Z}^+.$

\begin{lemma}\label{lema}
	\begin{enumerate}
    \item\label{it1}  For $\varepsilon >0,$ we have 
		$$\displaystyle \sum_{\tau= k+1}^\infty \left( \frac{1}{\tau}\right)^{\varepsilon+1}\le \frac{1}{\varepsilon k^\varepsilon}.$$
	\item\label{it2} For $0<\varepsilon<1,$ we have 
    $$(k+1)^\varepsilon \le \varepsilon\displaystyle \sum_{\tau=1}^k \tau^{\varepsilon-1} +1, ~k\in \mathbb{Z}^+.$$
	\end{enumerate}
\end{lemma}
\begin{proof}
\begin{enumerate}
   \item  Take $r= -\varepsilon (<0)$ then for $x=\tau, ~y= \tau-1, ~\tau\in \mathbb{Z}^+,$ Lemma I gives

	$$ -\varepsilon(\tau-1)^{-\varepsilon-1}\le \tau^{-\varepsilon} - (\tau-1)^{-\varepsilon} \le -\varepsilon \tau^{-\varepsilon-1}, $$
	i.e.,
	$$  -\varepsilon(\tau-1)^{-(\varepsilon+1)}\le \Delta(\tau-1)^{-\varepsilon} \le -\varepsilon \tau^{-(\varepsilon+1)}. $$
	
	From the RHS estimate of the above inequality, we get
	$$ \displaystyle \sum_{\tau= k+1}^\infty \Delta (\tau-1)^{-\varepsilon} \le -\varepsilon \displaystyle \sum_{\tau= k+1}^\infty \tau^{-(\varepsilon+1)}, ~k\in \mathbb{Z}^+.  $$
	On simplifying it, we have
   $$ \sum_{\tau= k+1}^\infty \left(\frac{1}{\tau}\right)^{\varepsilon+1} \le \frac{1}{\varepsilon k^\varepsilon}. $$
	
    \item  In Lemma I, on taking $r=\varepsilon, ~x=\tau+1, ~y= \tau, ~\tau\in \mathbb{Z}^+,$ we get
	$$ \varepsilon (\tau+1)^{\varepsilon-1} \le (\tau+1)^{\varepsilon}-\tau^\varepsilon \le \varepsilon\tau^{\varepsilon-1},$$
	i.e., 
	$$ \varepsilon (\tau+1)^{\varepsilon-1} \le \Delta\tau^\varepsilon \le \varepsilon\tau^{\varepsilon-1}.$$
	For $k\in \mathbb{Z}^+,$ RHS estimate of the above inequality gives
	$$\sum_{\tau=1}^k \Delta \tau^\varepsilon \le \varepsilon\sum_{\tau=1}^k \tau^{\varepsilon-1},~k\in \mathbb{Z}^+, $$
	i.e., $$ (k+1)^\varepsilon  \le \varepsilon\sum_{\tau=1}^k \tau^{\varepsilon-1}+1.$$

\end{enumerate}
	
 \end{proof}

 \begin{lemma}\label{lm34}
Let $ p>0, ~\beta > -1$ and $\{w(k)\}_{k\in \mathbb{Z}^+} \in \mathcal{QB}_{\beta,p}$ with constant $c$. Then there exists $\varepsilon, \,\, 0<\varepsilon<  \frac{1}{2(c+1) \max\lge \frac{1}{p+\beta p}, 1\rge}$ such that $\{w(k)\}\in \mathcal{QB}_{\beta,p-\varepsilon}.$
\end{lemma}
\begin{proof} 
Let $\{w(k)\}_{k\in \mathbb{Z}^+}\in \mathcal{QB}_{\beta,p}.$ By definition of the weight class $\mathcal{QB}_{\beta,p},$
we have the following estimate 
\begin{equation}\label{m1}
\sum_{k=1}^n\left(\frac{k}{n}\right)^{\beta p}w(k) + \sum_{k=n}^\infty \left(\frac{n}{k}\right)^{ p}w(k)\le (c+1)\sum_{k=1}^n\left(\frac{k}{n}\right)^{\beta p}w(k).
\end{equation}
Also, we have the following estimate
\begin{align}\label{m2} \nonumber\sum_{k=i}^\infty\left(\frac{n}{k}\right)^{ p+\beta p} \frac{1}{k}
&=\left(\frac{n}{i}\right)^{ p+\beta p} \frac{1}{i}+\sum_{k=i+1}^\infty\left(\frac{n}{k}\right)^{ p+\beta p} \frac{1}{k}\le  \left(\frac{n}{i}\right)^{ p+\beta p} \frac{1}{i} +\sum_{k=i+1}^\infty n^{ p+\beta p} \int_{k-1}^k\frac{1}{t^{p+\beta p+1}}dt\\\nonumber
&\le  \left(\frac{n}{i}\right)^{ p+\beta p} \frac{1}{i} + n^{ p+\beta p} \int_{i}^\infty \frac{1}{t^{p+\beta p+1}}dt
= \left(\frac{n}{i}\right)^{ p+\beta p} \frac{1}{i} + \frac{1}{p+\beta p}\left(\frac{n}{i}\right)^{ p+\beta p}  \\
&= \left(\frac{n}{i}\right)^{ p+\beta p} \left(\frac{1}{i} + \frac{1}{p+\beta p}\right). \end{align}
By using the Fubini, the  estimates \eqref{m1} and \eqref{m2}, we get the following
\begin{align}\label{m5}
\nonumber\sum_{k=n}^\infty\left(\frac{n}{k}\right)^{ p+\beta p} \frac{1}{k}\sum_{i=1}^k \left(\frac{i}{n}\right)^{ \beta p}w(i)
&\le\sum_{k=n}^\infty\left(\frac{n}{k}\right)^{ p+\beta p} \frac{1}{k}\sum_{i=1}^n \left(\frac{i}{n}\right)^{ \beta p}w(i)\\\nonumber 
&+\sum_{i=n}^\infty\left(\sum_{k=i}^\infty \left(\frac{n}{k}\right)^{ p+\beta p} \frac{1}{k}\right) \left(\frac{i}{n}\right)^{ \beta p}w(i)\\\nonumber
&\le \left(\frac{1}{n}+\frac{1}{p+\beta p}\right)\sum_{i=1}^n \left(\frac{i}{n}\right)^{ \beta p}w(i)\\\nonumber
&+  \sum_{i=n}^\infty \left(\frac{1}{i}+\frac{1}{p+\beta p}\right)\left(\frac{n}{i}\right)^{ p+\beta p}
\left(\frac{i}{n}\right)^{ \beta p} w(i)
\\\nonumber
&\le 2\max\lge\frac{1}{p+\beta p},1\rge \lr\sum_{i=1}^n\lr \frac{i}{n}\rr^{ \beta p}w(i)
+  \sum_{i=n}^\infty\left(\frac{n}{i}\right)^{ p} w(i)\rr\\
&\le 2(c+1)\max\lge\frac{1}{p+\beta p},1 \rge\sum_{i=1}^n\left(\frac{i}{n}\right)^{\beta p}w(i).
\end{align}
The estimate \eqref{m5} can also be written as 
\begin{equation}\label{m8}
n^{p+\beta p}\sum_{k=n}^\infty\frac{1}{k^{ p+\beta p+1}} \sum_{i=1}^k i^{ \beta p}w(i)\le 2C\sum_{k=1}^n k^{\beta p}w(k),
\end{equation}
where $C=(c+1)\max\lge\frac{1}{p+\beta p},1 \rge.$\\
Let $Tf(n)=n^{p+\beta p}\displaystyle\sum_{k=n}^\infty\frac{1}{k^{ p+\beta p+1}}f(k)$, where $f(k)=\displaystyle\sum_{i=1}^k i^{\beta p}w(i).$ 
So that,  the inequality \eqref{m8} gives
\[Tf(n) \le 2C f(n),~~ \text{for all} ~~n. \]
On iterating the above inequality $m$-times, we obtain 
\[\underbrace{T\circ\cdots\circ T}_{m} \left(\sum_{i=1}^n i^{\beta p}w(i)\right)\le \lr 2C\rr^m\sum_{k=1}^n k^{\beta p}w(k) \]
\begin{equation}\label{m9}
n^{p+\beta p}\sum_{k=n}^\infty\ln\left(\frac{k}{n}\right)^m (m!)^{-1}\frac{1}{k^{ p+\beta p+1}} \sum_{i=1}^k i^{ \beta p}w(i)\le (2C)^m\sum_{k=1}^n k^{\beta p}w(k).
\end{equation}
Choose $\varepsilon>0$ such that  $ \varepsilon 2 C< 1.$ Then using \eqref{m9} we obtain
\begin{align}\label{10m}
\nonumber n^{p+\beta p}\sum_{k=n}^\infty\left(\frac{k}{n}\right)^\varepsilon \frac{1}{k^{ p+\beta p+1}} \sum_{i=1}^k i^{ \beta p}w(i) & = n^{p+\beta p}\sum_{k=n}^\infty\exp\left(\varepsilon \ln\left(\frac{k}{n}\right)\right) \frac{1}{k^{ p+\beta p+1}} \sum_{i=1}^k i^{ \beta p}w(i)\\\nonumber
 &=  n^{p+\beta p}\sum_{k=n}^\infty \sum_{m=0}^\infty \frac{\varepsilon^m \left( \ln\left(\frac{k}{n}\right)\right)^m}{m!}\frac{1}{k^{ p +\beta p +1}} \sum_{i=1}^k i^{ \beta p}w(i)\\\nonumber
&\le \sum_{m=0}^\infty  (\varepsilon 2C)^{m} \sum_{k=1}^n k^{ \beta p}w(k)\\
&= \frac{1}{1-\varepsilon2C}\sum_{k=1}^n k^{\beta p}w(k). 
\end{align}
Now, by using the following estimate
    \[k^{ -p -\beta p +\varepsilon} \lesssim \sum_{i=k}^\infty \frac{1}{i^{ p+\beta p-\varepsilon+1}},\]
Fubini theorem and \eqref{10m}, we get
    \begin{align*}
    n^{p+\beta p-\varepsilon}\sum_{k=n}^\infty 
    \frac{1}{k^{ p-\varepsilon}} w(k) & \lesssim n^{p+\beta p-\varepsilon}
    \sum_{k=n}^\infty \left(\sum_{i=k}^\infty \frac{1}{i^{ p+\beta p-\varepsilon+1}}\right) k^{\beta p}w(k) \\
    & \lesssim n^{p+\beta p-\varepsilon}\sum_{i=n}^\infty  
    \frac{1}{i^{ p+\beta p-\varepsilon+1}}\sum_{k=1}^i  k^{\beta p}w(k) \\
        &\lesssim \sum_{k=1}^n k^{\beta p}w(k)\\
    &\lesssim n^{\beta\varepsilon} \sum_{k=1}^n k^{\beta( p-\varepsilon)}w(k).
\end{align*}
Thus, we obtain
\begin{align*}\sum_{k=n}^\infty \left(\frac{n}{k}\right)^{ p-\varepsilon} w(k) 
    &\lesssim n^{\beta(\varepsilon-p) } \sum_{k=1}^n k^{\beta( p-\varepsilon)}w(k)\\ &  \lesssim\sum_{k=1}^n \lr\frac{k}{n}\rr^{\beta( p-\varepsilon)}w(k). 
    \end{align*}
Hence we have proved that $$\{w(k)\}_{k\in \mathbb{Z}^+}\in \mathcal{QB}_{\beta,p-\varepsilon}.$$
\end{proof}
\vspace{0.1cm} 
\begin{remark}
Lemma \ref{lm34} is an improvement of Theorem 2.6 \cite{sa}.  In Lemma \ref{lm34}, on taking $\beta=0$ yields  one of the important properties of the  Ari\~{n}o -Muckenhoupt class of weights $\mathcal B_p.$ In  \cite{sa}, authors  proved the Lemma in this case under an additional assumption that the  weight sequence is non-increasing. Here we have proved it without such an assumption. Below we state the Lemma \ref{lm34} explicitly  for $\beta=0$ i.e., for the $\mathcal B_p$ class of weights.
\end{remark}
\begin{lemma}\label{la41}
Let $ p>0$ be given and $\{w(k)\}_{k\in \mathbb{Z}^+} \in \mathcal{B}_{p}$ with constant $c$. Then there exists $\varepsilon, \,\, 0<\varepsilon<  \frac{1}{2(c+1) \max\lge \frac{1}{p}, 1\rge}$ such that $\{w(k)\}_{k\in \mathbb{Z}^+}\in \mathcal{B}_{p-\varepsilon}.$
\end{lemma}
 
 \begin{lemma}\label{lemb}
 	Let $\beta \ge 0, ~p_0>0$ and $0<\varepsilon<p_0(\beta +1).$ Then
 	\begin{equation*}
 		n^{p_0(1+\beta)-\varepsilon} \le \tilde{c} \sum_{k=1}^n k^{p_0(1+\beta)-\varepsilon-1},
 	\end{equation*}
 	where $\tilde{c}:= \max\{1, ~p_0(1+\beta)-\varepsilon\}.$
 \end{lemma}
 \begin{proof}
  The proof follows from Power rule I on taking  $q= p_0(\beta+1)- \varepsilon$.
 	\end{proof}
 Below we give an important intermediary required to prove our main result on extrapolation.
 \begin{proposition}\label{lemc}
 	Let $\phi$ be a non-decreasing function defined on $(0,\infty),~ \beta \ge 0$ and $p_0>0$ be given. Suppose that for all weight sequences $\{w(k)\}_{k\in \mathbb{Z}^+}$ in $\mathcal{QB}_{\beta, p_0}$ the following holds
 	\begin{equation*}\label{eqa}
 		\sum_{k=1}^\infty f(k)w(k) \le \phi \lr[w]_{\mathcal{QB}_{\beta, p_0}}\rr\sum_{k=1}^\infty g(k)w(k),
 	\end{equation*}
 	for non-negative sequences $\{f(k)\}_{k\in \mathbb{Z}^+}$ and $\{g(k)\}_{k\in \mathbb{Z}^+}.$ Then for every $0<\varepsilon<p_0(\beta +1),$ we have the following
 	\begin{equation*}
 		\sum_{k=1}^n f(k)k^{p_0-1-\varepsilon} \le \phi(c(\beta, p_0, \varepsilon))\sum_{k=1}^n g(k)k^{p_0-1-\varepsilon},  ~ n\in \mathbb{Z}^+,
 	\end{equation*} 
 	where $c(\beta, p_0, \varepsilon):= \left(1+\frac{1}{\varepsilon}\right)\max\{p_0(1+\beta)-\varepsilon, 1\}.$
 \end{proposition}
\begin{proof}
	Suppose $\{f(k)\}_{k\in \mathbb{Z}^+}$ and $\{g(k)\}_{k\in \mathbb{Z}^+}$ are non-negative sequences. Take a particular type of weight sequence  $w(k)= v(k)k^{p_0-1-\varepsilon},$ where $\{v(k)\}_{k\in \mathbb{Z}^+}$ is a non-negative non-increasing sequence. We shall prove that $\{w(k)\}_{k\in \mathbb{Z}^+}\in \mathcal{QB}_{\beta, p_0}.$
	Using the fact that $\{v(k)\}_{k\in \mathbb{Z}^+}$ is non-increasing and applying Lemma \ref{lemb},  we have
	\begin{align*}
		n^{p_0(1+\beta)}\sum_{k=n}^\infty w(k)\left(\frac{1}{k}\right)^{p_0}
		& = n^{p_0(1+\beta)}\sum_{k=n}^\infty \left(\frac{1}{k}\right)^{p_0}k^{p_0-1-\varepsilon}v(k) \nonumber\\
		& = n^{p_0(1+\beta)}\sum_{k=n}^\infty k^{-(1+\varepsilon)}v(k) \\
		& \le n^{p_0(1+\beta)}v(n)\left(\sum_{k=n+1}^\infty \left(\frac{1}{k} \right)^{1+\varepsilon}+\left(\frac{1}{n} \right)^{1+\varepsilon} \right) \\
		& \le n^{p_0(1+\beta)}v(n)\left( \frac{1}{\varepsilon n^\varepsilon} + \frac{1}{n^{1+\varepsilon}}\right) \nonumber \\
		& = n^{p_0(1+\beta)-\varepsilon}v(n)\left( \frac{1}{\varepsilon} + \frac{1}{n}\right) \\
		& \le n^{p_0(1+\beta)-\varepsilon}v(n)\left( 1+ \frac{1}{\varepsilon} \right) \\
		& \le \tilde{c}\left( 1+ \frac{1}{\varepsilon} \right)v(n)\sum_{k=1}^n k^{p_0(1+\beta)-\varepsilon-1} \\
		& \le \tilde{c}\left( 1+ \frac{1}{\varepsilon} \right)\sum_{k=1}^n k^{p_0\beta}k^{p_0-1-\varepsilon}v(k) \\
		& = c(\beta, p_0, \varepsilon) \sum_{k=1}^n k^{p_0\beta}w(k),
	\end{align*}
	where $c(\beta, p_0, \varepsilon) := \left( 1+ \frac{1}{\varepsilon} \right)\tilde{c}= \left( 1+ \frac{1}{\varepsilon} \right)\max\{p_0(1+\beta)-\varepsilon, 1\}.$
	Therefore $\{w(k)\}_{k\in \mathbb{Z}^+}\in \mathcal{QB}_{\beta, p_0},$ with constant $c(\beta, p_0, \varepsilon),$ for every non-negative non-increasing sequence $\{v(k)\}_{k\in \mathbb{Z}^+}.$ Thus, in particular, on taking $v(k)= \{\chi_n(k)\}_{k\in \mathbb{Z}^+},$ for $\{w(k):= k^{p_0-\varepsilon-1}\chi_n(k)\}_{k\in \mathbb{Z}^+},$ which is in $\mathcal{QB}_{\beta, p_0}$ as shown above, by  (\ref{eqa}) we have 
	
	$$\sum_{k=1}^\infty f(k)k^{p_0-\varepsilon-1}\chi_n (k) \le \phi \lr[w]_{\mathcal{QB}_{\beta, p_0}}\rr \sum_{k=1}^\infty g(k)k^{p_0-\varepsilon-1}\chi_n (k). $$
	Since $[w]_{\mathcal{QB}_{\beta, p_0}} \le c(\beta, p_0, \varepsilon)$ and $\phi$ is a non-decreasing function, the above gives
	\[
	\sum_{k=1}^n f(k)k^{p_0-\varepsilon-1} \le \phi (c(\beta, p_0, \varepsilon)) \sum_{k=1}^n g(k)k^{p_0-\varepsilon-1}, ~n\in \mathbb{Z}^+.
	\]
	Hence, the proof is done.
\end{proof}
\begin{theorem}
	Let $\phi$ be a non-negative, non-decreasing function defined on $(0,\infty).$ Let $p_0 \ge 2$ and $\beta \ge 0$ be given. Suppose that for every $\{w(k)\}_{k\in \mathbb{Z}^+} \in \mathcal{QB}_{\beta, p_0},$ we have the following 
	\begin{equation*}
		\sum_{k=1}^\infty f^{p_0}(k)w(k) \le \phi \lr[w]_{\mathcal{QB}_{\beta, p_0}}\rr\sum_{k=1}^\infty g^{p_0}(k)w(k)
	\end{equation*}
	for non-negative sequences $\{f(k)\}_{k\in \mathbb{Z}^+}, \{g(k)\}_{k\in \mathbb{Z}^+} \in \mathcal{Q}_\beta .$ Then, for every  $p \ge p_0$ and every $\{w(k)\}_{k\in \mathbb{Z}^+} \in \mathcal{QB}_{\beta, p},$ the following inequality holds:
	\begin{equation}\label{eq23}
		\sum_{k=1}^\infty f^p(k)w(k) \le \tilde{\phi}(p_0, \beta, \phi) \sum_{k=1}^\infty g^p(k)w(k),
	\end{equation}
	where $\tilde{\phi}(p_0, \beta, \phi):= C \displaystyle \inf_{0<\varepsilon\le p_0-1}c(p_0, \beta, \varepsilon),$ and $C$ is a positive constant.
\end{theorem}
\begin{proof}
	Let $ p\ge  p_0$ and $ \{w(k)\}_{k\in \mathbb{Z}^+} \in \mathcal{QB}_{\beta, p}.$  Suppose that $\{f(k)\}_{k\in \mathbb{Z}^+}$ and $\{g(k)\}_{k\in \mathbb{Z}^+}$ are non-negative quasi non-increasing sequences. Since $\{f(k)\}_{k\in \mathbb{Z}^+} \in \mathcal{Q}_\beta,$ this implies that $\{h(k)= k^{-\beta}f(k)\}_{k\in \mathbb{Z}^+}$ is a non-increasing sequence. 
	
	\noindent Let $0<\varepsilon \le p_0-1,$ Using the fact that $h(k)$ is non-increasing and Lemma \ref{lemb}, we may estimate $h^{p_0}(k) $  as
	\begin{align}
		h^{p_0}(k) 
		&= \frac{k^{p_0(1+\beta)-\varepsilon}}{k^{p_0(1+\beta)-\varepsilon}}h^{p_0}(k) \nonumber \\
		& \le \frac{p_0(1+\beta)-\varepsilon}{k^{p_0(1+\beta)-\varepsilon}}\left( \sum_{\tau=1}^k \tau^{p_0(\beta+1)\varepsilon-1}\right)h^{p_0}(k) \nonumber \\
		& \le \frac{p_0(1+\beta)-\varepsilon}{k^{p_0(1+\beta)-\varepsilon}}\left( \sum_{\tau=1}^k \tau^{p_0(\beta+1)\varepsilon-1}h^{p_0}(\tau)\right) \label{eqx}.
	\end{align}
	Now, considering the LHS of \eqref{eq23}, by using \eqref{eqx}, Proposition \ref{lemc} and the Power rule I, we get 
	\begin{align}
		 \nonumber \sum_{k=1}^\infty f^p(k)w(k)
		& = \sum_{k=1}^\infty k^{\beta p}h^p(k)w(k)  \\\nonumber
		& = \sum_{k=1}^\infty \left( k^{\beta p_0}h^{p_0}(k) \right)^{p/p_0} \\\nonumber
		& \le \sum_{k=1}^\infty \lr \frac{p_0(1+\beta)-\varepsilon}{k^{p_0(1+\beta)-\varepsilon}}k^{\beta p_0} \left( \sum_{\tau=1}^k \tau^{p_0(1+\beta)-\varepsilon-1}h^{p_0}(\tau) \right) \rr^{p/p_0}w(k) \nonumber \\
		& = \left( p_0(1+\beta)-\varepsilon \right)^{p/p_0}\sum_{k=1}^\infty\lr \frac{1}{k^{p_0-\varepsilon}}\sum_{\tau=1}^k \tau^{p_0-\varepsilon-1}h^{p_0}(\tau)\tau^{p_0\beta}  \rr^{p/p_0}w(k)  \nonumber \\
		& = \left( p_0(1+\beta)-\varepsilon \right)^{p/p_0}\sum_{k=1}^\infty\lr\frac{1}{k^{p_0-\varepsilon}}\sum_{\tau=1}^k \tau^{p_0-\varepsilon-1}f^{p_0}(\tau)  \rr^{p/p_0}w(k)\nonumber \\\nonumber
        &\le \left( p_0(1+\beta)-\varepsilon \right)^{p/p_0}\left(\phi\left(\left(p_0(\beta+1)-\varepsilon \right)\left( 1+\frac{1}{\varepsilon}\right)\right)\right)^{p/p_0} \times\\
		& \hspace{1cm}\sum_{k=1}^\infty\lr \frac{1}{k^{p_0-\varepsilon}}\sum_{\tau=1}^k \tau^{p_0-\varepsilon-1}g^{p_0}(\tau)  \rr^{p/p_0}w(k) \nonumber \\
		& \le\lr\left( p_0(1+\beta)-\varepsilon \right)\max\{p_0-\varepsilon,1\}\phi\left(\left(p_0(\beta+1)-\varepsilon \right)\left( 1+\frac{1}{\varepsilon}\right)\right) \rr^{p/p_0} \times \nonumber\\\nonumber
		& \hspace{1cm} \sum_{k=1}^\infty\lr \frac{1}{\Psi(k)}\sum_{\tau=1}^k g^{p_0}(\tau)\psi(\tau) \rr^{p/p_0}w(k)\\
        &=c(p_0, p, \beta, \varepsilon)\sum_{k=1}^\infty \left( \mathcal{A}_\psi g^{p_0} \right)^{p/p_0}(k)w(k), \label{hdy}        
        \end{align}
	where $\psi(\tau)= \tau^{p_0-\varepsilon-1}$ and 
    $c(p_0, p, \beta, \varepsilon)= \lr\left( p_0(1+\beta)-\varepsilon \right)\max\{p_0-\varepsilon,1\}\phi\left(\left(p_0(\beta+1)-\varepsilon \right)\left( 1+\frac{1}{\varepsilon}\right)\right) \rr^{p/p_0}.$
     \vspace{3pt} 
 
Since $\{w(k)\}_{k\in \mathbb{Z}^+}  \in \mathcal{QB}_{\beta, p},$  by Lemma \ref{lm34} there exists $\sigma, ~0<\sigma<   \frac{1}{2~[w]_{\mathcal{QB}_{\beta, p}} \max\lge \frac{1}{p+\beta p}, 1\rge}$ such that $\{w(k)\}_{k\in \mathbb{Z}^+} \in \mathcal{QB}_{\beta,p-\sigma}.$ Since \eqref{hdy} holds for any $\varepsilon, ~ 0< \varepsilon \le p_0-1,$ we can choose $\varepsilon>0$ such that $\{w(k)\}_{k\in \mathbb{Z}^+} \in \mathcal{QB}_{\beta, (p_0-\varepsilon)\frac{p}{p_0}}.$
I.e.,
\begin{equation*}
\sum_{k=n}^\infty \left( \frac{n}{k}\right)^{(p_0-\varepsilon)\frac{p}{p_0}}w(k) \le C\sum_{k=1}^n \left( \frac{k}{n}\right)^{\beta(p_0-\varepsilon)\frac{p}{p_0}}w(k).
\end{equation*}
Or,
\begin{equation*}
n^{(p_0-\varepsilon+\beta)\frac{p}{p_0}}\sum_{k=n}^\infty \left( \frac{1}{k}\right)^{(p_0-\varepsilon)\frac{p}{p_0}}w(k) \le C \sum_{k=1}^n k^{\frac{\beta p}{p_0}}\left( \frac{k}{n}\right)^{\beta(p_0-\varepsilon-1)\frac{p}{p_0}}w(k).
\end{equation*}
Since $p_0-\varepsilon-1>0,$ we have
\begin{equation}
 n^{(p_0-\varepsilon+\beta)\frac{p}{p_0}}\sum_{k=n}^\infty \left( \frac{1}{k}\right)^{(p_0-\varepsilon)\frac{p}{p_0}}w(k) \le C \sum_{k=1}^n k^{\frac{\beta p}{p_0}}w(k) \label{eqx4}
\end{equation}
Now, the Power rule I gives
$$k^{p_0-\varepsilon} \le (p_0-\varepsilon)\sum_{k=1}^n k^{p_0-\varepsilon-1} = (p_0-\varepsilon)\Psi(k)$$ 
and
$$\sum_{k=1}^n k^\beta \psi(k) = \sum_{k=1}^n k^{p_0-\varepsilon+\beta-1} \le n^{p_0-\varepsilon+\beta}.$$
Therefore, \eqref{eqx4} can be written as 
\begin{equation*}\label{eqy3}
	\left( \sum_{k=1}^n k^\beta \psi(k) \right)^{\frac{p}{p_0}}\sum_{k=n}^\infty \Psi^{-\frac{p}{p_0}}(k)w(k) \le C(p_0-\varepsilon)^{\frac{p}{p_0}}\sum_{k=1}^n k^{\frac{\beta p}{p_0}}w(k).
\end{equation*}
Again, by Power rule I, we have
$$n^\beta \displaystyle\sum_{k=1}^n \psi(k) \le \lr\beta+p_0-\varepsilon\rr\displaystyle\sum_{k=1}^n k^\beta \psi(k). $$ \\
Hence in view of the Remark \ref{r3} and Theorem \ref{th1}, we have
\begin{equation}\label{eqy4}
	\sum_{k=1}^\infty \left( \left( \mathcal{A}_\psi g^{p_0}\right)(k)\right)^{p/p_0}w(k) \le C \sum_{k=1}^\infty g^p(k) w(k).
\end{equation}
On combining (\ref{hdy}) and (\ref{eqy4}), we have
\[
\sum_{k=1}^\infty f^p(k)w(k) \le c(p_0, p, \beta, \varepsilon)C \sum_{k=1}^\infty g^p(k)w(k).
\]
Setting $\tilde{\phi}(p_0, p, \beta, \varepsilon)= C \displaystyle\inf_{0<\varepsilon\le p_0-1}c(p_0, p, \beta, \varepsilon),$ so that the above inequality becomes
\[
\sum_{k=1}^\infty f^p(k)w(k) \le \tilde{\phi}(p_0, \beta, \phi) \sum_{k=1}^\infty g^p(k)w(k),
\]
which completes the proof.
\end{proof}

\bigskip

\noindent{\it Acknowledgment.} 
The research of Monika Singh was supported by the  National Board of Higher Mathematics, 
research project no. 02011/14/2023 NBHM(R.P)/R\&D II/5951, India.

\noindent The research  of Amiran Gogatishvili was partially supported by the grant project 23-04720S of the Czech Science Foundation, The Institute of Mathematics, CAS is supported  by RVO:67985840, by  Shota Rustaveli National Science Foundation (SRNSF), grant no: FR22-17770, and by the
grant Ministry of Education and Science of the Republic of Kazakhstan (project no.
AP14869887). 
\medskip

\noindent {\it Conflict of interest statement.} The authors state that there is no conflict of interest.


\end{document}